\documentclass[graybox]{svmult}

\usepackage{type1cm}        
\usepackage{makeidx}         
\usepackage{graphicx}        
\usepackage{multicol}        
\usepackage[bottom]{footmisc}

\usepackage{newtxtext}       %
\usepackage[varvw]{newtxmath}       

\usepackage{pgfplots}
\usepackage{comment}
\usepackage{tikz}
\newcommand{\Int}{\mathrm{int}}

\makeindex             

\begin{document}

\title*{Cross-points in the Neumann-Neumann method}
\author{Bastien Chaudet-Dumas and Martin J. Gander}
\institute{Bastien Chaudet-Dumas \at University of Geneva, Switzerland \email{bastien.chaudet@unige.ch}
\and Martin J. Gander \at University of Geneva, Switzerland \email{martin.gander@unige.ch}}
%
%
\maketitle

\section{Introduction}
\label{sec:1}

The Neumann-Neumann method (NNM), first introduced in \cite{Bourgat:1989:VFA} in the case of two subdomains, is among the most popular non-overlapping domain decomposition methods.  However, when used as a stationary solver at the continuous level, it has been observed that the method faced well-posedness issues in the presence of cross-points, see \cite{chaouqui2018local}. Here, our goal is to analyze in detail the behaviour of the NNM near cross-points on a simple, but rather instructive, bidimensional configuration.

Let $\Omega\subset \mathbb{R}^2$ be the square $(-1,1)\times(-1,1)$, divided into four non-overlapping square subdomains $\Omega_i$, $i\in\mathcal{I}:=\{1,2,3,4\}$, see Figure \ref{fig:SchDD}. This leads to one interior cross-point (red dot), and four boundary cross-points (black dots). We denote the interfaces between adjacent subdomains by $\Gamma_{ij}:= \Int(\partial\Omega_i\cap\partial\Omega_j)$, the skeleton of the partition by $\Gamma:=\bigcup_{i,j} \overline{\Gamma}_{ij}$, and $\partial\Omega_i^0:=\partial\Omega_i\cap \partial\Omega$.
We consider the Laplace problem with Dirichlet boundary conditions on $\Omega$, that is: find $u$ solution to
\begin{equation}
        -\Delta u = f \: \mbox{ in } \Omega, \quad u = g\: \mbox{ on } \partial\Omega,
    \label{eqn:DirPb}
\end{equation}
where $f\in L^2(\Omega)$ and $g\in H^{\frac{3}{2}}(\partial\Omega)$, ensuring that $u\in H^2(\Omega)$.
%
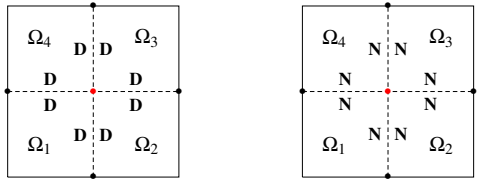
\begin{figure}
\centering
\resizebox{0.22\textwidth}{!}{
    \begin{tikzpicture}

    \draw[black, thin] (-2,-2) -- (2,-2) -- (2,2) -- (-2, 2) -- (-2,-2);
    \draw[black, densely dashed] (0,-2) -- (0,2); 
    \draw[black, densely dashed] (-2,0) -- (2,0); 
    
    
    \node[] at (0,-2) {\normalsize{$\bullet$}};
    \node[] at (0,2) {\normalsize{$\bullet$}};
    \node[] at (-2,0) {\normalsize{$\bullet$}};
    \node[] at (2,0) {\normalsize{$\bullet$}};
    \node[red] at (0,0) {\normalsize{$\bullet$}};
    
    \node[] at (-1.25,-1.25) {\Large{$\Omega_1$}};
    \node[] at (1.25,-1.25) {\Large{$\Omega_2$}};
    \node[] at (1.25,1.25) {\Large{$\Omega_3$}};
    \node[] at (-1.25,1.25) {\Large{$\Omega_4$}};

    \node[] at (-0.3,-1) {\large{\textbf{D}}};
    \node[] at (-1,-0.3) {\large{\textbf{D}}};
    \node[] at (0.3,1) {\large{\textbf{D}}};
    \node[] at (1,0.3) {\large{\textbf{D}}};
    \node[] at (0.3,-1) {\large{\textbf{D}}};
    \node[] at (-1,0.3) {\large{\textbf{D}}};
    \node[] at (-0.3,1) {\large{\textbf{D}}};
    \node[] at (1,-0.3) {\large{\textbf{D}}};
    
    \end{tikzpicture}
}
\hspace{4em}
\resizebox{0.22\textwidth}{!}{
    \begin{tikzpicture}

    \draw[black, thin] (-2,-2) -- (2,-2) -- (2,2) -- (-2, 2) -- (-2,-2);
    \draw[black, densely dashed] (0,-2) -- (0,2); 
    \draw[black, densely dashed] (-2,0) -- (2,0); 
    
    \node[] at (0,-2) {\normalsize{$\bullet$}};
    \node[] at (0,2) {\normalsize{$\bullet$}};
    \node[] at (-2,0) {\normalsize{$\bullet$}};
    \node[] at (2,0) {\normalsize{$\bullet$}};
    \node[red] at (0,0) {\normalsize{$\bullet$}};
    
    \node[] at (-1.25,-1.25) {\Large{$\Omega_1$}};
    \node[] at (1.25,-1.25) {\Large{$\Omega_2$}};
    \node[] at (1.25,1.25) {\Large{$\Omega_3$}};
    \node[] at (-1.25,1.25) {\Large{$\Omega_4$}};

    \node[] at (-0.3,-1) {\large{\textbf{N}}};
    \node[] at (-1,-0.3) {\large{\textbf{N}}};
    \node[] at (0.3,1) {\large{\textbf{N}}};
    \node[] at (1,0.3) {\large{\textbf{N}}};
    \node[] at (0.3,-1) {\large{\textbf{N}}};
    \node[] at (-1,0.3) {\large{\textbf{N}}};
    \node[] at (-0.3,1) {\large{\textbf{N}}};
    \node[] at (1,-0.3) {\large{\textbf{N}}};
    
    \end{tikzpicture}
}
  \caption{Transmission conditions of the standard NNM for $u$ (left) and $\psi$ (right).}
  \label{fig:SchDD}
\end{figure}

Given an initial couple $(u^0,\psi^0)$, and a relaxation parameter $\theta\in \mathbb{R}$, each iteration $k\geq 1$ of the NNM applied to \eqref{eqn:DirPb} can be split into two steps:
\begin{itemize}
    \item \textit{(Dirichlet step)} Solve for all $i\in \mathcal{I}$,
    \begin{equation*}
        \begin{aligned}
            -\Delta u_i^k  &= f \: \mbox{ in } \Omega_i \:, \quad
            u_i^k = g \: \mbox{ on } \partial\Omega_i^0 \:, \\
            u_i^k &= u_i^{k-1}-\theta\left( \psi_i^{k-1}+\psi_j^{k-1} \right)\: \mbox{ on } \Gamma_{ij}, \: \forall j\in \mathcal{I}\: \mbox{ s.t. } \Gamma_{ij}\neq\emptyset\:.
        \end{aligned}
    \end{equation*}
    \item \textit{(Neumann step)} Compute the correction $\psi^k$, that is, solve for all $i\in \mathcal{I}$,
    \begin{equation*}
        \begin{aligned}
            -\Delta \psi_i^k &= 0 \: \mbox{ in } \Omega_i \:, \quad
            \psi_i^k = 0 \: \mbox{ on } \partial\Omega_i^0 \:, \\
            \partial_{n_i}\psi_i^k &= \partial_{n_i}u_i^k+\partial_{n_j}u_j^k\: \mbox{ on } \Gamma_{ij}, \: \forall j\in \mathcal{I}\: \mbox{ s.t. } \Gamma_{ij}\neq\emptyset \:.
        \end{aligned}
    \end{equation*}
\end{itemize}
For the method to be well defined, it is assumed in the rest of this paper that the initial couple $(u^0,\psi^0)$ is compatible with the Dirichlet boundary condition, i.e. it satisfies: $u^0\in H^2(\Omega)$, $\psi^0\in H^2(\Omega)\cap H^1_0(\Omega)$ and $u^0\mid_{\partial\Omega\cap\Gamma} = g\mid_\Gamma$.

\section{Convergence analysis of the Neumann-Neumann method}

\begin{definition}
  A measurable function $h:\Omega\to\mathbb{R}$ is said to be \textit{even symmetric} (resp. \textit{odd symmetric}) if for a.e. $(x,y) \in \Omega$, $h(-x,-y)=h(x,y)$ (resp. $-h(x,y)$). Moreover, any measurable function $h$ can be uniquely decomposed into $h=h_e+h_o$ where $h_e$ is even symmetric and $h_o$ is odd symmetric.
\end{definition}
Following this notion, as in \cite{chaudet2022cross1}, we introduce the so-called \textit{even symmetric} and \textit{odd symmetric parts} of problem \eqref{eqn:DirPb}: find $u_e$ and $u_o$
solutions to
\begin{subequations}\label{eq:EvenOddDirPb}
    \begin{align}
        -\Delta u_e = f_e \: \mbox{ in } \Omega,& \quad  u_e = g_e\: \mbox{ on } \partial\Omega, \label{subeqn:EvenDirPb} \\
        -\Delta u_o = f_o \: \mbox{ in } \Omega,& \quad  u_o = g_o\: \mbox{ on } \partial\Omega. \label{subeqn:OddDirPb}
    \end{align}
\end{subequations}
If $u$ denotes the solution to \eqref{eqn:DirPb}, it is known (see \cite{chaudet2022cross1}) that the unique solutions $u_e$ and $u_o$ to 
these subproblems are precisely the even symmetric part and the odd symmetric part of $u$. In what follows, we will perform the convergence analysis of the NNM separately for the errors associated with the even and odd symmetric subproblems, as they lead to completely different behaviours of the method. \\

\textbf{Case of the even symmetric part.}
The next Theorem states that the NNM is convergent when applied to the even symmetric part of \eqref{eqn:DirPb}.

\begin{theorem}
    Taking $(u^0_e,\psi^0_e)$ as initial couple for the NNM applied to \eqref{subeqn:EvenDirPb} produces a sequence $\left\{u_e^k\right\}_k$ that converges geometrically to the solution $u_e$ with respect to the $L^2$-norm and the broken $H^1$-norm for any $\theta\in (0,\frac{1}{2})$. Moreover, the convergence factor is given by $|1-4\theta|$, which also proves that the method becomes a direct solver for the specific choice $\theta=\frac{1}{4}$.
    \label{thm:GeoCvgEvenDirNN}
\end{theorem}

\begin{proof}
    As in \cite{chaudet2022cross1} for the Dirichlet-Neumann method, let us study the first iterations of the NNM in terms of the local errors $e_{e,i}^k:=u_e|_{\Omega_i}-u_{e,i}^k$.

    \noindent $\bullet$ \textit{Iteration $k=1$, Dirichlet step:} 
    In each $\Omega_i$, $i\in \mathcal{I}$, the errors satisfy
            \begin{equation*}
                \begin{aligned}
                    -\Delta e_{e,i}^1 &= 0 \: \mbox{ in } \Omega_i \:, \quad
                    e_{e,i}^1 = 0 \: \mbox{ on } \partial\Omega_i^0 \:, \\
                    e_{e,i}^1 &= e_{e,i}^0+\theta\left( \psi_{e,i}^0+\psi_{e,j}^0\right)  \: \mbox{ on } \Gamma_{ij}, \: \forall j\in \mathcal{I}\: \mbox{ s.t. } \Gamma_{ij}\neq\emptyset \:.
                \end{aligned}
            \end{equation*}
            Since $(u_e^0,\psi_e^0)$ is compatible with the even symmetric part of the Dirichlet boundary condition, $e_{e,i}^1$ exists and is unique in $H^1(\Omega_i)$. Using the even symmetry properties of $e_e^0$ and $\psi_e^0$, one can deduce that the $e_{e,i}^1$, for $i\in\{2,3,4\}$, can be expressed in terms of $e_{e,1}^1$ as follows:
            \begin{equation*}
                \begin{aligned}
                    e_{e,2}^1(x,y) &= e_{e,1}^1(-x,y)\:, & \mbox{ for a.e. } (x,y) \in\Omega_2 \:, \\
                    e_{e,3}^1(x,y) &= e_{e,1}^1(-x,-y)\:, & \mbox{ for a.e. } (x,y) \in \Omega_3 \:, \\
                    e_{e,4}^1(x,y) &= e_{e,1}^1(x,-y)\:, & \mbox{ for a.e. } (x,y) \in \Omega_4 \:.
                \end{aligned}
            \end{equation*}

            \noindent $\bullet$ \textit{Iteration $k=1$, Neumann step:} 
            We compute the correction $\psi_{e,i}^1$ in each subdomain $\Omega_i$. For instance, taking $i=1$, we get in $\Omega_1$
            \begin{equation*}
                \begin{aligned}
                    -\Delta \psi_{e,1}^1 &= 0 \: \mbox{ in } \Omega_1 \:, \quad
                    \psi_{e,1}^1 = 0 \: \mbox{ on } \Gamma_1 \:, \\
                    \partial_{n_1}\psi_{e,1}^1 &= -\left( \partial_{n_1}e_{e,1}^1+\partial_{n_2}e_{e,2}^1\right) = -2\partial_{n_1} e_{e,1}^1 \: \mbox{ on } \Gamma_{12} \:, \\
                    \partial_{n_1}\psi_{e,1}^1 &= -\left(\partial_{n_1}e_{e,1}^1+\partial_{n_4}e_{e,4}^1\right) = -2\partial_{n_1} e_{e,1}^1\: \mbox{ on } \Gamma_{41} \:.
                \end{aligned}
            \end{equation*}
            Thus, uniqueness of $\psi_{e,1}^1$ in $H^1(\Omega_1)$ yields $\psi_{e,1}^1=-2e_{e,1}^1$ in $\Omega_1$. A similar reasoning applies to each $\psi_{e,i}^1$, $i\in\{2,3,4\}$, therefore the recombined correction simply reads: $\psi_e^1 = -2e_e^1$ in $\Omega\setminus\Gamma$.
        
        \noindent $\bullet$ \textit{Iteration $k\geq 2$:}
        At iteration $k=2$, the transmission condition for the Dirichlet step in $\Omega_i$ on each $\Gamma_{ij}$ is given by,
        $
            e_{e,i}^2 = e_{e,i}^1+\theta\left( \psi_{e,i}^1+\psi_{e,j}^1\right) = (1-4\theta) e_{e,i}^1\:.
        $
        Uniqueness of $e_{e,i}^2$ in $H^1(\Omega_i)$ enables us to conclude that $e_{e,i}^2=(1-4\theta) e_{e,i}^1$ in $\Omega_i$. Since this holds in each subdomain, the exact same reasoning as for iteration $k=1$ applies, and we get after the Neumann step
        $
            e_e^2 = (1-4\theta) e_e^1 
        $ and 
        $
            \psi_e^2 = -2(1-4\theta) e_e^1
        $
        in $\Omega\setminus \Gamma$.
        By induction, we obtain for any $k\geq 3$,
        $
            e_e^k = (1-4\theta)^{k-1} e_e^1
        $
        in $\Omega\setminus \Gamma$.
        This leads to the following estimates for the error on the whole domain $\Omega$ in the $L^2$-norm and the broken $H^1$-norm:
        $$
            \parallel u_e^k-u_e \parallel_{L^2(\Omega)} = \sum_{i\in \mathcal{I}} \parallel e_{e,i}^k \parallel_{L^2(\Omega_i)} \leq C|1-4\theta|^{k-1} \:, \\
        $$
        $$
            \sum_{i\in \mathcal{I}} \parallel u_{e,i}^k-u_{e,i} \parallel_{H^1(\Omega_i)} \leq C'|1-4\theta|^{k-1} \:,
        $$
        where $C$, $C'$ are strictly positive constants depending on the data and the geometry of the domain decomposition.

\end{proof}

\textbf{Case of the odd symmetric part.}
As for the Dirichlet-Neumann method, the NNM does not converge in general when applied to the odd symmetric part of \eqref{eqn:DirPb}.

\begin{theorem}
	The NNM applied to \eqref{subeqn:OddDirPb} is not well-posed. More specifically, taking $(u_o^0,\psi_o^0)$ as initial couple, there exists an integer $k_0>0$ such that the solution to the problem obtained at the $k_0$-th iteration is not unique. In addition, all possible solutions $u_o^{k_0}$ are singular at the cross-point, with a leading singularity of type $(\ln r)^2$.
    \label{thm:DivgOddDirDN}    
\end{theorem}
\begin{theorem}
	If we let the NNM go beyond the ill-posed iteration $k_0$ from Theorem \ref{thm:DivgOddDirDN}, we end up with a sequence $\{u_o^k\}_{k\geq k_0}$ of non-unique iterates. Moreover, for each $k\geq k_0$, all possible $u_o^k$ are singular at the cross-point, with a leading singularity of type $(\ln r)^{2(k-k_0)+2}$.    
    \label{thm:PropSingOddDirDN}    
\end{theorem}
\begin{proof}
    The proofs of these results rely on the exact same arguments as those in the proofs of \cite[Theorem 7 and 8]{chaudet2022cross1}. 
\end{proof}

The previous results show that, at some point in the iterative process, the NNM method will lead to solving an ill-posed problem. This will generate a singular solution, and the generated singularity will then propagate through the following iterations.

\section{Toward a modified Neumann-Neumann method}

The conclusions from the previous section suggest that the transmission conditions of the standard NNM are naturally well adapted to the even symmetric part of the problem. Indeed, in this context, one may express at each iteration $k$ all local errors $e_{e,i}^k$ in terms of only one, say $e_{e,1}^k$, by symmetry. This motivates the search for different transmission conditions such that a similar symmetry property holds for the odd symmetric part of the problem. \\

\textbf{Fixing the odd symmetric case.}
In order to fix the well-posedness issue in the odd symmetric case, and obtain the symmetry property mentioned above, we propose a new distribution of Dirichlet and Neumann transmission conditions, as shown in Figure \ref{fig:SchDDMix}.

\begin{figure}
\centering
\resizebox{0.22\textwidth}{!}{
    \begin{tikzpicture}

    \draw[black, thin] (-2,-2) -- (2,-2) -- (2,2) -- (-2, 2) -- (-2,-2);
    \draw[black, densely dashed] (0,-2) -- (0,2); 
    \draw[black, densely dashed] (-2,0) -- (2,0); 
    
    
    \node[] at (0,-2) {\normalsize{$\bullet$}};
    \node[] at (0,2) {\normalsize{$\bullet$}};
    \node[] at (-2,0) {\normalsize{$\bullet$}};
    \node[] at (2,0) {\normalsize{$\bullet$}};
    \node[red] at (0,0) {\normalsize{$\bullet$}};
    
    \node[] at (-1.25,-1.25) {\Large{$\Omega_1$}};
    \node[] at (1.25,-1.25) {\Large{$\Omega_2$}};
    \node[] at (1.25,1.25) {\Large{$\Omega_3$}};
    \node[] at (-1.25,1.25) {\Large{$\Omega_4$}};

    \node[] at (-0.3,-1) {\large{\textbf{N}}};
    \node[] at (-1,-0.3) {\large{\textbf{D}}};
    \node[] at (0.3,1) {\large{\textbf{N}}};
    \node[] at (1,0.3) {\large{\textbf{D}}};
    \node[] at (0.3,-1) {\large{\textbf{N}}};
    \node[] at (-1,0.3) {\large{\textbf{D}}};
    \node[] at (-0.3,1) {\large{\textbf{N}}};
    \node[] at (1,-0.3) {\large{\textbf{D}}};
    
    \end{tikzpicture}
}
\hspace{4em}
\resizebox{0.22\textwidth}{!}{
    \begin{tikzpicture}

    \draw[black, thin] (-2,-2) -- (2,-2) -- (2,2) -- (-2, 2) -- (-2,-2);
    \draw[black, densely dashed] (0,-2) -- (0,2); 
    \draw[black, densely dashed] (-2,0) -- (2,0); 
    
    \node[] at (0,-2) {\normalsize{$\bullet$}};
    \node[] at (0,2) {\normalsize{$\bullet$}};
    \node[] at (-2,0) {\normalsize{$\bullet$}};
    \node[] at (2,0) {\normalsize{$\bullet$}};
    \node[red] at (0,0) {\normalsize{$\bullet$}};
    
    \node[] at (-1.25,-1.25) {\Large{$\Omega_1$}};
    \node[] at (1.25,-1.25) {\Large{$\Omega_2$}};
    \node[] at (1.25,1.25) {\Large{$\Omega_3$}};
    \node[] at (-1.25,1.25) {\Large{$\Omega_4$}};

    \node[] at (-0.3,-1) {\large{\textbf{D}}};
    \node[] at (-1,-0.3) {\large{\textbf{N}}};
    \node[] at (0.3,1) {\large{\textbf{D}}};
    \node[] at (1,0.3) {\large{\textbf{N}}};
    \node[] at (0.3,-1) {\large{\textbf{D}}};
    \node[] at (-1,0.3) {\large{\textbf{N}}};
    \node[] at (-0.3,1) {\large{\textbf{D}}};
    \node[] at (1,-0.3) {\large{\textbf{N}}};
    
    \end{tikzpicture}
}
  \caption{Transmission conditions of the mixed NNM for $u$ (left) and $\psi$ (right).}
  \label{fig:SchDDMix}
\end{figure}
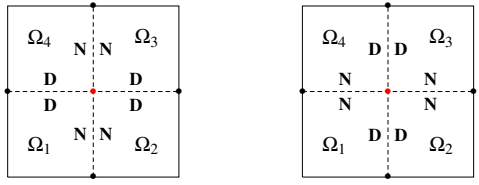
Let us introduce $\Gamma_D^1$, $\Gamma_N^1$, $\Gamma_D^2$, $\Gamma_N^2$ the sets containing all parts of the interface $\Gamma$ where transmission conditions of Dirichlet or Neumann type are imposed for $u$ (superscript 1) and for $\psi$ (superscript 2), that is :
\begin{equation*}
    \Gamma_D^1 := \{ \Gamma_{23}, \Gamma_{41} \}, \quad \Gamma_N^1 := \{ \Gamma_{12}, \Gamma_{34} \}, \quad
    \Gamma_D^2 := \{ \Gamma_{12}, \Gamma_{34} \}, \quad \Gamma_N^2 := \{ \Gamma_{23}, \Gamma_{41} \}.
\end{equation*}
Given an initial couple $(u^0,\psi^0)$ and relaxation parameter $\theta$, each iteration $k\geq 1$ of the proposed \textit{mixed} Neumann-Neumann method can be split into two steps:
\begin{itemize}
    \item  \textit{(First step)}  Solve for all $i\in \mathcal{I}$
    \begin{equation*}
        \begin{aligned}
            -\Delta u_i^k  &= f \: \mbox{ in } \Omega_i \:, \quad
            u_i^k = g \: \mbox{ on } \partial\Omega_i^0 \:, \\
            u_i^k &= u_i^{k-1} - \theta\left(\psi_i^{k-1}+\psi_j^{k-1}\right)\: \mbox{ on } \Gamma_{ij}, \: \forall j\in \mathcal{I} \: \mbox{ s.t. } \Gamma_{ij}\in\Gamma_D^1 \:, \\
            \partial_{n_i}u_i^k &= \partial_{n_i}u_i^{k-1} + (-1)^i \theta\left(\partial_{n_i}\psi_i^{k-1} +\partial_{n_j}\psi_j^{k-1}\right) \: \mbox{ on } \Gamma_{ij}, \: \forall j\in \mathcal{I} \: \mbox{ s.t. } \Gamma_{ij}\in\Gamma_N^1 \:.
        \end{aligned}
    \end{equation*}
        \item  \textit{(Second step)} Compute the correction $\psi^k$, that is, solve for all $i\in \mathcal{I}$
    \begin{equation*}
        \begin{aligned}
            -\Delta \psi_i^k  &= 0 \: \mbox{ in } \Omega_i \:, \quad
            \psi_i^k = 0 \: \mbox{ on } \partial\Omega_i^0 \:, \\
            \psi_i^k &= u_i^k-u_j^k \: \mbox{ on } \Gamma_{ij}, \: \forall j\in \mathcal{I} \: \mbox{ s.t. } \Gamma_{ij}\in\Gamma_D^2 \:, \\
            \partial_{n_i}\psi_i^k &= \partial_{n_i}u_i^k+\partial_{n_j}u_j^k\: \mbox{ on } \Gamma_{ij}, \: \forall j\in \mathcal{I} \: \mbox{ s.t. } \Gamma_{ij}\in\Gamma_N^2 \:.
        \end{aligned}
    \end{equation*}
\end{itemize}

With this choice of transmission conditions, we are able to prove that the proposed mixed NNM is convergent when applied to the odd symmetric part of \eqref{eqn:DirPb}.
\begin{theorem}
    Taking $(u^0_o,\psi^0_o)$ as initial couple for the mixed NNM applied to \eqref{subeqn:OddDirPb} produces a sequence $\left\{u_o^k\right\}_k$ that converges geometrically to the solution $u_o$ with respect to the $L^2$-norm and the broken $H^1$-norm for any $\theta\in (0,\frac{1}{2})$. Moreover, the convergence factor is given by $|1-4\theta|$, which also proves that the method becomes a direct solver for the specific choice $\theta=\frac{1}{4}$.
    \label{thm:GeoCvgOddDirNN}
\end{theorem}
\begin{proof}
    We follow the same steps as in the proof of Theorem \ref{thm:GeoCvgEvenDirNN}.

    \noindent $\bullet$ \textit{Iteration $k=1$, Dirichlet step:} 
    In each $\Omega_i$, $i\in \mathcal{I}$, the odd errors satisfy
            \begin{equation*}
                \begin{aligned}
                    -\Delta e_{o,i}^1 &= 0 \: \mbox{ in } \Omega_i \:, \quad
                    e_{o,i}^1 = 0 \: \mbox{ on } \partial\Omega_i^0 \:, \\
                    e_{o,i}^1 &= e_{o,i}^0+\theta\left( \psi_{o,i}^0+\psi_{o,j}^0\right)  \: \mbox{ on } \Gamma_{ij}, \: \forall j\in \mathcal{I}\: \mbox{ s.t. } \Gamma_{ij}\in\Gamma_D^1 \:, \\
                    \partial_{n_i}e_{o,i}^1 &= \partial_{n_i}e_{o,i}^{0} - (-1)^i\theta\left(\partial_{n_i}\psi_{o,i}^{0} +\partial_{n_j}\psi_{o,j}^{0}\right) \: \mbox{ on } \Gamma_{ij}, \: \forall j\in \mathcal{I} \: \mbox{ s.t. } \Gamma_{ij}\in\Gamma_N^1 \:.
                \end{aligned}
            \end{equation*}
            These problems are well-posed since $(u_o^0,\psi_o^0)$ is compatible with the odd symmetric part of the boundary condition. This time, using the mixed conditions enforced along $\Gamma$ together with the odd symmetry properties of $e_o^0$ and $\psi_o^0$, we can deduce that
            \begin{equation*}
                \begin{aligned}
                    e_{o,2}^1(x,y) &= -e_{o,1}^1(-x,y)\:, & \mbox{ for a.e. } (x,y) \in\Omega_2 \:, \\
                    e_{o,3}^1(x,y) &= -e_{o,1}^1(-x,-y)\:, & \mbox{ for a.e. } (x,y) \in \Omega_3 \:, \\
                    e_{o,4}^1(x,y) &= e_{o,1}^1(x,-y)\:, & \mbox{ for a.e. } (x,y) \in \Omega_4 \:.
                \end{aligned}
            \end{equation*}
            Indeed, for the first equality, taking $(x,y)\in\Omega_2$, we have on $\Gamma_{23}$ and $\Gamma_{12}$
            \begin{equation*}
            \begin{aligned}
                e_{o,2}^1(x,0) &= e_{o,2}^0(x,0)+\theta\left( \psi_{o,2}^0(x,0)+\psi_{o,3}^0(x,0)\right)  \\
                &= -e_{o,1}^0(-x,0)-\theta\left( \psi_{o,4}^0(-x,0)+\psi_{o,1}^0(-x,0)\right) = -e_{o,1}^1(-x,0)\:, \\
                (\partial_{n_2} e_{o,2}^1)(0,y) &= -(\partial_x e_{o,2}^0)(0,y)-\theta\left( (\partial_x\psi_{o,2}^0)(0,y)+(\partial_x\psi_{o,1}^0)(0,y)\right)  \\
                &= -(\partial_x e_{o,1}^0)(0,y)-\theta\left( (\partial_x\psi_{o,1}^0)(0,y)+(\partial_x\psi_{o,2}^0)(0,y)\right) \\
                &= -(\partial_{n_1} e_{o,1}^1)(0,y) = -(\partial_{n_2} e_{o,1}^1(-\,\cdot,\cdot))(0,y)\:.
            \end{aligned}
            \end{equation*}
            Then uniqueness of the solution to the subproblem in $\Omega_2$ yields $e_{o,2}^1= -e_{o,1}^1(-\,\cdot,\cdot)$ a.e. in $\Omega_2$. The two other equalities are obtained using similar arguments, see Figure \ref{fig:ErrMNNSymmetry} for an illustration of this symmetry property.   

            \noindent $\bullet$ \textit{Iteration $k=1$, Neumann step:} 
            For $i=1$, we get in $\Omega_1$
            \begin{equation*}
                \begin{aligned}
                    -\Delta \psi_{o,1}^1 &= 0 \: \mbox{ in } \Omega_1 \:, \quad
                    \psi_{o,1}^1 = 0 \: \mbox{ on } \Gamma_1 \:, \\
                    \psi_{o,1}^1 &= -e_{o,1}^1+e_{o,2}^1 = -2 e_{o,1}^1\: \mbox{ on } \Gamma_{12} \:, \\
                    \partial_{n_1}\psi_{o,1}^1 &= -\left( \partial_{n_1}e_{o,1}^1+\partial_{n_4}e_{o,4}^1\right) = -2\partial_{n_1} e_{o,1}^1 \: \mbox{ on } \Gamma_{41} \:.
                \end{aligned}
            \end{equation*}
            Therefore $\psi_{o,1}^1=-2e_{o,1}^1$ in $\Omega_1$. Extending these arguments to the other subdomains yields a recombined correction $\psi_o^1 = -2e_o^1$ in $\Omega\setminus\Gamma$.
        
        \noindent $\bullet$ \textit{Iteration $k\geq 2$:}
        At iteration $k=2$, the transmission conditions for the first step in $\Omega_1$ are given by
        \begin{equation*}
        \begin{aligned}
             e_{o,1}^2 &= e_{o,1}^1+\theta\left( \psi_{o,1}^1+\psi_{o,4}^1\right) = (1-4\theta) e_{o,1}^1 \: \mbox{ on } \Gamma_{41}\:, \\
             \partial_{n_1}e_{o,1}^2 &= \partial_{n_1}e_{o,1}^{1} + \theta\left(\partial_{n_1}\psi_{o,1}^{1} +\partial_{n_2}\psi_{o,2}^{1}\right) = (1-4\theta)\partial_{n_1}e_{o,1}^{1} \: \mbox{ on } \Gamma_{12}\:.
        \end{aligned}
        \end{equation*}
        This implies that $e_{o,1}^2=(1-4\theta) e_{o,1}^1$ in $\Omega_1$. Using the same arguments in the other subdomains and performing the second step leads to
        $
            e_o^2 = (1-4\theta) e_o^1
        $
        and $\psi_o^2 = -2(1-4\theta) e_o^1$ in $\Omega\setminus \Gamma$.
        As in the proof of Theorem \ref{thm:GeoCvgEvenDirNN}, we obtain by induction that, for any $k\geq 3$,
        $
            e_o^k = (1-4\theta)^{k-1} e_o^1
        $
        in $\Omega\setminus \Gamma$.
        The desired error estimates are then deduced from the last relation.
        
\end{proof}

\textbf{The new NNM.}
Here are the different steps of our \textit{new} NNM to solve \eqref{eqn:DirPb} starting from an initial couple $(u^0,\psi^0)$ compatible with the Dirichlet boundary condition, and a relaxation parameter $\theta\in (0,1/2)$.
\begin{enumerate}
    \item Decompose the data into their even/odd symmetric parts to get \eqref{subeqn:EvenDirPb} and \eqref{subeqn:OddDirPb}.
    \item Solve in parallel:
    \begin{itemize}
        \item \eqref{subeqn:EvenDirPb} using the standard NNM starting from $(u_e^0,\psi_e^0)$,
        \item \eqref{subeqn:OddDirPb} using the mixed NNM starting from $(u_o^0,\psi_o^0)$.
    \end{itemize}
    \item Recompose the solution $u=u_e+u_o$.
\end{enumerate}

\begin{remark}
    It is actually enough to solve for $u_e$ and $u_o$ in $\Omega_1\cup\Omega_2$, and then extend them to the whole domain $\Omega$ by symmetry. One iteration of the new NNM thus costs the same as one iteration of the original NNM.
\end{remark}

\section{Numerical experiments}

In order to test our new NNM, we apply it to two simple benchmarks: one with even symmetric data (Example 1: $g=0$ and $f=1$) and one with odd symmetric data (Example 2: $g=0$ and $f=x+y+h$ where $h=\sin(2\phi)$ in $\Omega_1$, $h=-\sin(2\phi)$ in $\Omega_3$ and $h=0$ in $\Omega_2\cup\Omega_4$, with $\phi$ being the angle in polar coordinates, see Figure \ref{fig:ErrMNNSymmetry}). The discretization of \eqref{eqn:DirPb} is performed using a standard five point finite difference scheme on a cartesian grid of meshsize $h=0.01$.
%
%
When two Dirichlet conditions meet at a corner, the value of $g$ at this node is set to the average of the two values. In addition, when Dirichlet and Neumann conditions meet at a corner, we choose the Dirichlet one to be enforced at this node.
The results obtained show that the method behaves as predicted by Theorem \ref{thm:GeoCvgEvenDirNN} and Theorem \ref{thm:GeoCvgOddDirNN}. Indeed, for $\theta=\frac{1}{4}$, the method converges after two iterations, see the left column in Figure \ref{fig:ErrMNNEvenOdd}. And for $\theta\in(0,\frac12)$, $\theta\neq \frac14$, the method converges geometrically to the solution, see the right column in Figure \ref{fig:ErrMNNEvenOdd}.

\begin{figure}[t]
  \begin{center}
    \includegraphics[width=0.45\textwidth]{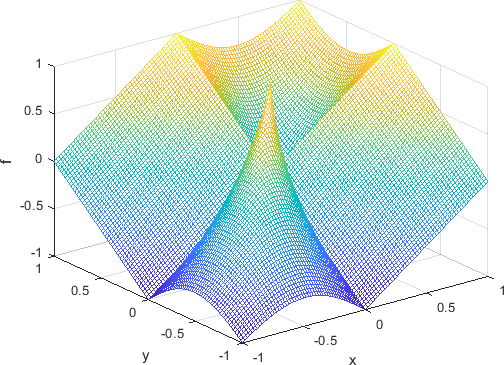}
    \hspace{1em}
    \includegraphics[width=0.45\textwidth]{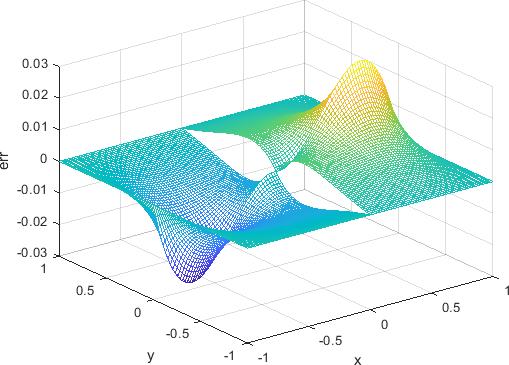}
  \end{center}
  \caption{Source term $f$ (left), and absolute error at iteration 1 for $\theta=0.25$ (right), in Example 2.}
  \label{fig:ErrMNNSymmetry}
\end{figure}

\begin{figure}[t]
  \begin{center}
    \includegraphics[width=0.5\textwidth]{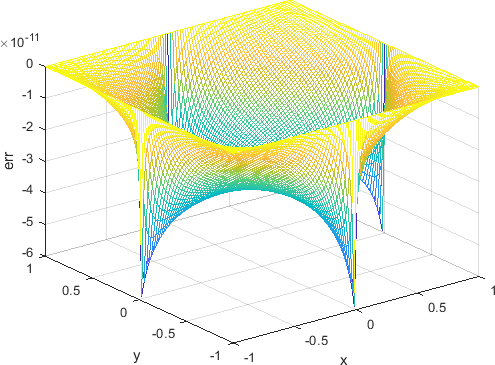}
    \hspace{1em}
    \includegraphics[width=0.45\textwidth]{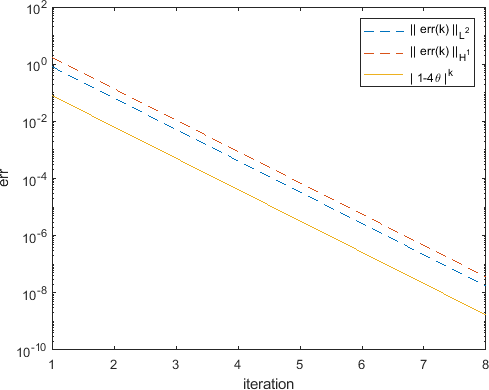}
    \includegraphics[width=0.5\textwidth]{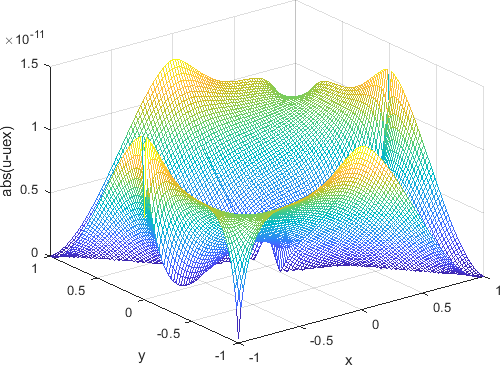}
    \hspace{1em}
    \includegraphics[width=0.45\textwidth]{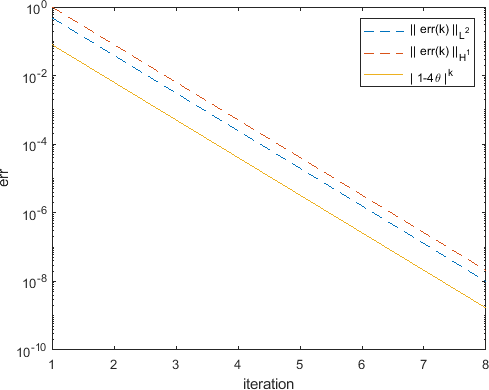}
    \end{center}
  \caption{Absolute error at iteration 2 for $\theta=0.25$ (left column), and error curve for $\theta=0.23$ (right column), in Example 1 (top) and Example 2 (bottom).}
  \label{fig:ErrMNNEvenOdd}
\end{figure}

In this short paper, we gave a complete analysis of the standard NNM in a simple configuration involving one cross-point. The even/odd decomposition showed that the NNM was able to treat very efficiently the even symmetric part of the solution, while it faced well-posedness and convergence issues when applied to the odd symmetric part of the solution.
Based on this observation, we proposed new mixed transmission conditions of Dirichlet/Neumann type to treat efficiently the odd symmetric part. We proved that the newly proposed NNM built upon a combination between the standard NNM and the new mixed method is convergent, and we validated this property by some numerical experiments. A natural extension of this work would be the 3D case of a cube divided into eight subcubes. It would also be interesting to generalize the notion of even/odd symmetry to the case of more general cross-points (not necessarily rectilinear, or involving a number of subdomains $N\neq 4$).

\bibliographystyle{plain}
\bibliography{bibliography}

\end{document}